\documentclass[11pt,twoside]{article}
\usepackage{amsmath}
\usepackage{amsthm}
\usepackage{amssymb}
\usepackage[dvipdfm]{hyperref}

\newcommand{\C}{\mathbb{C}}
\newcommand{\R}{\mathbb{R}}
\newcommand{\Z}{\mathbb{Z}}
\newcommand{\N}{\mathbb{N}}
\theoremstyle{plain}
\newtheorem{theorem}{Theorem}[section]
\newtheorem{proposition}[theorem]{Proposition}
\newtheorem{lemma}[theorem]{Lemma}
\newtheorem{Corollary}[theorem]{Corollary}

\theoremstyle{definition}

\newtheorem{remark}[theorem]{Remark}

\newtheorem{example}[theorem]{Example}

\makeatletter
\@addtoreset{equation}{section}
\makeatother

\begin{document}


\title{On the singular sets of solutions to the  
Kapustin--Witten equations and the 
Vafa--Witten ones on compact K\"{a}hler surfaces 
} 
\author{  
Yuuji Tanaka}
\date{}


\maketitle

\begin{abstract}
This article finds a structure of singular sets on compact K\"{a}hler surfaces, which Taubes introduced in the studies of the asymptotic analysis of solutions to the Kapustin--Witten equations and the Vafa--Witten ones  originally on smooth four-manifolds.   
These equations can be seen as real four-dimensional analogues of the Hitchin equations on Riemann surfaces, and one of common obstacles to be overcome is a certain unboundedness of solutions to these equations, especially of the "Higgs fields". 
The singular sets by Taubes describe part of the limiting behaviour of a sequence of solutions with this unboundedness property, and Taubes proved that the real two-dimensional Haussdorff measures of these singular sets are finite.  
In this article, we look into the singular sets, when the underlying manifold is a compact K\"{a}hler surface, and find out that they have the structure of an analytic subvariety in this case. 

\end{abstract}





\section{Introduction}

Two sets of gauge-theoretic equations that we consider in this article originated in $\mathcal{N}=4$ super Yang--Mills theory in Theoretical Physics; they appear after {\it topological twists} of the theory. 
However, we introduce these equations as 4-dimensional analogues of Hitchin's  equations on Riemann surfaces, 
since many ideas and techniques for the Hitchin equations seem to be enhanced for the studies of these equations.

\paragraph{Hitchin's equations on compact Riemann surfaces.} 

Let $\Sigma$ be a compact Riemann surface with genus greater than one, and let $E$ be a holomorphic vector bundle on $\Sigma$. Hitchin \cite{Hi} introduced the following equations seeking for a pair $(A, \Phi)$ consisting of a connection $A$ on $E$ and a section $\Phi$ of $\text{End}(E) \otimes \Lambda^{1,0}$, where $\Lambda^{1,0} := ( T^* X \otimes \C)^{1,0}$. 

\begin{gather}
\bar{\partial}_A \Phi =0,  \\
F_A + [ \Phi, \Phi^* ] = 0. 
\end{gather}

The outcome of the studies of the Hitchin equations has been wonderfully abundant, 
so perhaps one may like to look at similar equations in higher dimensions, say, on a complex surface. 
In fact, the two sets of equations that we consider in this article are ones of them.   
In that case, since $\Lambda^{1,0}$ can be thought of as either the holomorphic cotangent bundle $\Omega^1_{\Sigma}$ or the canonical bundle $K_{\Sigma}$ on a Riemann surface, 
there could be at least two possibilities of analogues of the Hitchin equations for a complex surface $X$: 
one is obtained by looking at a section of $\text{End}(E) \otimes \Omega^1_X$;  
and the other one considers a section of $\text{End}(E) \otimes K_X$. 
The former is the Kapustin--Witten equations, and the latter is the  Vafa--Witten equations as described below.

\paragraph{The Kapustin--Witten equations on closed four-manifolds.} 

Let $X$ be a compact, oriented, Riemannian four-manifold, and let $P \to X$
be a principal $SO(3)$ or $SU(2)$ bundle over $X$. 
We denote by $\mathfrak{g}_{P}$ the adjoint bundle of $P$, and by
$\mathcal{A}_{P}$ the space of connections on $P$. 
By using the Riemannian metric, we decompose the space of two-forms as
$\Omega^2 (X) = \Omega^{+} (X) \oplus \Omega^{-} (X)$. 
We call elements in $\Omega^{+} (X)$ {\it self-dual two-forms}, and ones in
$\Omega^{-}(X)$ {\it anti-self-dual two-forms}.

We consider the following version of the {\it Kapustin--Witten
equations}, that ask a pair $(A, \mathfrak{a}) \in \mathcal{A}_{P}
\times \Gamma ( X, \mathfrak{g}_{P} \otimes T^{*} X )$ to satisfy    
\begin{gather}
d_{A}^{*} \mathfrak{a} =0 , \quad ( d_{A} \mathfrak{a} )^{-} =0  ,  
\label{eq:KW1} \\
F_{A}^{+} - [\mathfrak{a} \wedge \mathfrak{a} ]^{+} =0, 
\label{eq:KW2}
\end{gather}
where $F_{A}$ is the curvature two-form of the connection $A$, and
superscripts ``-'' and ``+'' indicate taking the anti-self-dual part or
self-dual part respectively. 
These equations were introduced by Kapustin and Witten \cite{KW} (see also \cite{Wi}) in the
context of the geometric Langlands programme from the viewpoint of
$\mathcal{N}=4$ super Yang--Mills theory in four dimensions. 
See also a paper by Gagliardo and Uhlenbeck \cite{GU}. 
Note that the equation $d_{A}^{*} \mathfrak{a} =0$ makes \eqref{eq:KW1}
and \eqref{eq:KW2} an elliptic
system after a gauge fixing equation. 
Note also that there are no solutions to the equations \eqref{eq:KW1}
and \eqref{eq:KW2} in the case when $X$ is compact if the first
Pontrjagin number of $\mathfrak{g}_{P}$ is positive (see \cite[\S
3.3]{KW}, \cite[\S 2.b)]{T2}).

\begin{remark}
In \cite{KW}, Kapustin and Witten originally obtained the following family of equations parametrized by $\tau \in \mathbb{P}^2$ from a topological twist of $\mathcal{N}=4$ super Yang--Mills theory in 4 dimensions.  
\begin{gather}
( F_{A} - [\mathfrak{a} \wedge \mathfrak{a} ] + \tau d_{A} \mathfrak{a} )^{+} =0,  
\label{eq:pKW1} \\
( F_{A} - [\mathfrak{a} \wedge \mathfrak{a} ] - \tau^{-1} d_{A} \mathfrak{a} )^{-} =0, 
\label{eq:pKW2} \\
d_{A}^{*} \mathfrak{a} =0 . 
\label{eq:pKW3}  
\end{gather}
The equations \eqref{eq:KW1} and \eqref{eq:KW2} are the case for $\tau =0$ in the above.  
The set of equations \eqref{eq:pKW1}, \eqref{eq:pKW2} and \eqref{eq:pKW3} for $\tau = \{ \infty \}$ is the ``orientation reversed'' one to the equations \eqref{eq:KW1} and \eqref{eq:KW2}. 
Kapustin and Witten deduced that there are no solutions to the equations \eqref{eq:pKW1}, \eqref{eq:pKW2} and \eqref{eq:pKW3} for $\tau \in \R \setminus \{ 0 \}$ on a compact 4-manifold with the first Pontrjagin class of $\mathfrak{g}_{P}$ being non zero; 
and for these cases solutions to the equations \eqref{eq:pKW1}, \eqref{eq:pKW2} and \eqref{eq:pKW3} on compact four-manifolds are flat $SL(2, \C)$ connections. 
Note that, if $\tau$ is not real, the equations \eqref{eq:pKW1}, \eqref{eq:pKW2} and \eqref{eq:pKW3} are overdetermined. 
Kapustin and Witten presented the relevance of the values $\tau = \pm 1, \pm i$ to the geometric Langlands programme in \cite{KW}. 
For the recent progress for the equations for $\tau = -1$ with the Nahm pole  boundary conditions, see 
\cite{He1}, \cite{He2}, \cite{HM}, \cite{MW1}, \cite{MW2} and \cite{LT}. 
\end{remark}

\paragraph{Taubes' analysis on $SL(2, \C)$ connections and the singular sets.}

The analytic properties of solutions to the above equations were successfully
revealed by Taubes. 
In \cite{T2}, Taubes studied the Uhlenbeck style compactness problem for
$SL(2,\C)$-connections, including solutions to the above equations, on
four-manifolds (see also \cite{T1}, \cite{T3}). 
One of the major difficulties to be overcome here was the lack of a priori boundedness of the 
``extra fields'' $\mathfrak{a}$. 
However, Taubes introduces some real codimension two singular set, to be denoted $Z$,
outside which a sequence of ``partially rescaled''
$SL(2,\C)$-connections 
converges after gauge transformations except bubbling
out at a finite set of points 
(see also \cite{D}, \cite{H}, \cite{HW}, \cite{Mo}, \cite{MSWW}, \cite{MSWW2}, \cite{Tak} and \cite{Tan1} for
related problems). 
We describe this briefly 
in Section \ref{sec:rtaubes}.

\paragraph{The Kapustin-Witten equations on compact K\"{a}hler surfaces and the singular sets.} 

Our interest in this article is the structure of the above singular set
$Z$. 
In fact, Taubes further studied the structure of the singular set $Z$
in the general setting \cite{T3} (see also \cite{Tak}), however, 
we restrict ourselves to the K\"{a}hler surface case.

To state our observation in this article, let $X$ be a compact K\"{a}hler surface with K\"{a}hler form $\omega$, and let $E$ be a Hermitian vector bundle on $X$ with Hermitian metric $h$. 
We denote by $\mathcal{A}_{(E,h)}$ the space of all unitary connections on $E$, 
and by $\mathfrak{u} (E) = \text{End}(E,h)$ the bundle of skew-Hermitian endomorphisms 
of $E$.  
Then the equations \eqref{eq:KW1} and \eqref{eq:KW2} take the following form for 
a pair $(A,\phi) \in \mathcal{A}_{(E,h)} \times \Omega^{1,0} (\mathfrak{u} (E)) $: 
\begin{gather}
\bar{\partial}_{A} \phi = 0, \quad [ \phi \wedge \phi ] =0, 
\label{eq:S1}\\
F_{A}^{0,2} = 0, \quad \Lambda \left( F_{A}^{1,1} + 2 [ \phi  \wedge \phi^{*}] \right) =0 , 
\label{eq:S2} 
\end{gather}
where $\Lambda$ is the adjoint of $\wedge \omega$.  
For the deduction of these, see \cite[\S 6(iii)]{N} or Section \ref{sec:eqK} to this article.  
These are the equations studied by Simpson in \cite{Si}. 
(Simpson considered these in arbitrary dimensions.)

We now assume the rank of $E$ to be two. 
We consider a sequence $\{ (A_n , \phi_n )\}$ of solutions to Simpson's equations \eqref{eq:S1} and \eqref{eq:S2}, and apply the above Taubes analysis in \cite{T2} to it, putting $r_n := || \phi_n ||_{L^2}$ and assuming that $\{ r_n \}$ has no bounded subsequence. 
We then obtain a singular set, to be denoted by $Z$, outside which the sequence $\{ (A_n , \phi_n / r_n \} $ has an $L^2_1$ convergent subsequence after a suitable choice of gauge transformations except bubbling out at a finite set of points in $X$. 
In this article, we prove the following.

\begin{theorem}[Corollary \ref{cor:hs}]
The singular set $Z$ has the structure of an analytic subvariety of $X$. 
\label{th:main}
\end{theorem}

This can be done by a simple observation that a singular set is 
identified with the zero set of a section of the holomorphic bundle $(\Omega_{X}^1)^{\otimes 2}$ on $X$ up to a finite set of points.

\paragraph{The Vafa--Witten equations on closed four-manifolds.}

We next consider another analogue of the Hitchin equations, called the {\it Vafa--Witten equations} \cite{VW}. 
The Vafa--Witten equations look for a triple $(A, B, \Gamma)$ consisting of a connection $A$ on a principal $G$-bundle with $G$ being a compact Lie group over a smooth oriented Riemannian four-manifold $X$; 
a section $B$ of the associated bundle $\mathfrak{g}_{P} \otimes \Lambda^{+}$;  
and a section $\Gamma$ of $\mathfrak{g}_{P}$, 
where $\mathfrak{g}_{P}$ is the adjoint bundle of $P$ and $\Lambda^+$ is the self-dual part of $\Lambda^2 T^*X$. 
The exact form of the equations is as follows. 
\begin{gather}
d_A^* B  + d_A \Gamma =0, 
\label{VW1} \\
F_A^+ + [B.B] + [B, \Gamma] =0,
\label{VW2}
\end{gather}
where $F_A^+$ is the self-dual part of the curvature of the connection $A$, and $[B.B] \in \Gamma (X, \mathfrak{g}_{P} \otimes \Lambda^+)$ is defined through the Lie brackets of $\mathfrak{g}_{P}$ and $\Lambda^+$ (see \cite[\S A.1.6]{M} or \cite[\S 2]{Tan1} for details).

Taubes \cite{T4} proved an analogous theorem to the Kapustin--Witten equations also for the Vafa--Witen equations case. 
In addition, our observation on the singular sets holds for the Vafa--Witten equations case, 
namely, the singular set can be identified with the zero set of a section of the square of the canonical bundle of $X$, so it also has a structure of an analytic subvariety of $X$ (Corollary \ref{cor2}).

\vspace{0.3cm}

The organization of this article is as follows.  
In Section \ref{sec:KW}, we deal with the Kapustin-Witten equations. 
Section \ref{sec:rtaubes} is a brief review of the results by Taubes in \cite{T2}. 
In Section \ref{sec:eqK}, we deduce that the Kapustin--Witten equations on a compact  K\"{a}hler surface are the same as Simpson's equations.  
Then Theorem \ref{th:main} is proved in Section \ref{sec:stsing}.   
We treat the Vafa--Witten equations in Section \ref{sec:VW}. 
In Sections \ref{sec:rtaubesVW}, we review results by Taubes \cite{T4} on the Vafa--Witten equations on smooth four-manifolds. 
We then consider the equations on compact K\"{a}hler surfaces, 
and prove a similar statement to Theorem \ref{th:main} for the Vafa--Witten case in Section \ref{sec:singVW}.

\paragraph{Acknowledgements.}
I would like to thank Cliff Taubes for drawing my attention to the Kapustin--Witten equations and  generously helpful suggestions about these.  
I have a huge debt to him for his insight and ingenious computations which had  tremendously inspired this article. 
I am grateful to Mikio Furuta, Siqi He, Tomoyuki Hisamoto, Hiroshi Iritani and Sakura Sch\"{a}fer-Nameki for useful discussions and valuable comments around the subject. 
I would also like to thank Ben Mares for helpful comments on the computations in Section \ref{sec:eqK}.  
I am also grateful to Seoul National University, NCTS at National Taiwan University, Kyoto University, BICMR at Peking University and Institut des Hautes \'{E}tudes Scientifiques for support and hospitality, where part of this work was done during my visits in 2015--17. 
This work was partially supported by JSPS Grant-in-Aid for Scientific Research Nos. 15H02054 and 16K05125; and a Simons Collaboration Grant on `Special holonomy in Geometry, Analysis and Physics'.

\section{The singular sets of solutions to the Kapustin--Witten equations}
\label{sec:KW}

\subsection{Results by Taubes}
\label{sec:rtaubes}

In this section, we briefly describe the work by Taubes on the Uhlenbeck
style compactness for $SL(2,
\C)$-connections on closed four-manifolds \cite{T2}. 
In \cite{T2}, Taubes proved the results described in this subsection in the framework of the equation \eqref{eq:pKW1}, \eqref{eq:pKW2} and \eqref{eq:pKW3} for $\tau  \in \R \cup \{ \infty \}$ on closed oriented Riemannian four-manifolds; and also on $M \times \mathbb{I}$, where $M$ is a closed oriented Riemannian three-manifold, and $\mathbb{I}$ is a closed bounded interval. 
Our description here is in the style for the case $\tau=0$, namely, for the equations \eqref{eq:KW1} and \eqref{eq:KW2}.

Let $X$ be a compact, oriented, Riemannian 4-manifold, and let $P \to X$
be a principal $G$-bundle over $X$. 
We take $G$ to be $SO(3)$ or $SU(2)$ in the below.  
We denote by $\mathfrak{g}_{P}$ the adjoint bundle of $P$.

We consider a sequence $\{ (A_{n} , \mathfrak{a}_{n}) \}_{n \in \N}$ 
of solutions to the equation \eqref{eq:KW1} and \eqref{eq:KW2}.  
We put $r_{n} := || \mathfrak{a}_{n}||_{L^2}$, and take $a_{n} :=
\mathfrak{a}_{n} / r_{n}$ for each $n \in \N$. 
Note that the pairs $(A_{n} , a_{n})$ satisfy the following. 
\begin{gather}
( d_{A_n} a_n )^{-} =0  , \quad d_{A_n}^{*} a_n =0 , 
\label{eq:rKW1} \\
F_{A_n}^{+} - r_n^2 [ a_n \wedge a_n ] ^{+} =0,
\label{eq:rKW2} 
\end{gather}
where $r_n >0$.

In the case that $\{ r_n \}$ has a bounded subsequence, the Uhlenbeck
compactness theorem with the Bochner--Weitzenb\"{o}ck formula deduces
the following.

\begin{proposition}[\cite{T2}]
If $\{ r_{n} \}$ has a bounded subsequence, then  
there exist a principal $G$-bundle $P_{\Delta} \to X$ and a pair 
$(A_{\Delta} , \mathfrak{a}_{\Delta})$, where $A_{\Delta}$ is a
 connection on $P_{\Delta}$ and $\mathfrak{a}_{\Delta}$ is a section
 of $\mathfrak{g}_{P} \otimes T^{*} X$, which satisfies the equations
 \eqref{eq:KW1} and \eqref{eq:KW2}, and a finite set $\Theta \subset X$;  
a subsequence $\Xi \subset \N$; and a sequence $\{ g_{n} \}_{n \in \Xi}$
 of automorphisms of $P_{\Delta}|_{X \setminus \Theta}$ such that 
the sequence $\{ (g_{n}^{*} A_{n} , g_{n}^{*} \mathfrak{a}_{n})\}_{ n \in \Xi}$ converges 
to $(A_{\Delta}
 , \mathfrak{a}_{\Delta})$ in the $C^{\infty}$-norm on compact subsets
 in $X \setminus \Theta$. 
\end{proposition}

Analysis for the case that $\{ r_{n} \}$ has no bounded subsequence was the
bulk of \cite{T2}. 
Firstly, a subsequence of $\{ | a_n | \}$ converges in $L^2_1$, we denote the limit  suggestively by $|\hat{a}_{\diamond}|$. 
Moreover, Taubes proves the following.

\begin{proposition}[\cite{T2}] 
There exists a finite set $\Theta \subset X$ such that 
the function $| \hat{a}_{\diamond} |$ is continuous on $X \setminus \Theta$,  
and it is smooth at points in $X \setminus \Theta$ where $| \hat{a}_{\diamond}|$ is positive. 
Furthermore, the sequence $\{ | a_{n} | \}_{n \in \Lambda \subset \N}$
 converges to $| \hat{a}_{\diamond} |$ in the $C^{0}$-topology on
 compact subsets in $X \setminus \Theta$. 
\label{prop:C0conv}
\end{proposition}

In the above proposition, the finite set $\Theta$  
corresponds to ``bubbling
points'' of the connections.  
We denote by $Z$ the zero locus of the limit function $| \hat{a}_{\diamond}|$. 
Then Taubes proves that 
a subsequence of $\{ (A_{n} , a_{n} )\}$ converges in $L^2_1$-topology outside
$\Theta \cup Z$. 
More precisely, Taubes proved the following.

\begin{theorem}[\cite{T2}]
There exist a real line bundle $\mathcal{I}$ over $X \setminus \{ \Theta \cup Z \}$, 
a section $\nu$ of $\mathcal{I} \otimes T^*X |_{X \setminus \{ \Theta \cup Z \}}$ with $d \nu = d^* \nu =0$ and $| \nu | = |\hat{a}_{\diamond}|$, 
a principal $G$-bundle $P_{\Delta}$ over $X \setminus \{ \Theta \cup Z \} $ 
and a connection $A_{\Delta}$ of $P_{\Delta}$ with $d_{A_{\Delta}} * F_{A_{\Delta}}=0$, 
an $A_{\Delta}$-covariantly constant homomorphism $\sigma_{\Delta} : \mathcal{I} \to \mathfrak{g}_{P_{\Delta}}$ 
and a sequence of isomorphisms $\{ g_n \}$ from $P_{\Delta}$ to $P|_{X \setminus \{ \Theta \cup Z \} } $ such that 
$\{ g_n (A_n), g_n (a_n) \}$ converges in $L^2_1$-topology outside $\Theta \cup Z$. 
Furthermore, $\{ g_n (a_n) \}$ converges in $C^0$-topology outside $\Theta$. 
\end{theorem}

As for the structure of the above $Z$, Taubes proved the following. 

\begin{proposition}[\cite{T3}]
$Z$ has the   Hausdorff dimension at most $2$. 
\end{proposition}

In Section \ref{sec:stsing}, we prove that the set $Z$ has a structure of an analytic set of $X$ when the underlying manifold is a K\"{a}hler surface.

\subsection{The equations on compact K\"{a}hler surfaces}
\label{sec:eqK}

From here, we take $X$ to be a compact K\"{a}hler surface with K\"{a}hler form $\omega$, and $E$ to be a Hermitian
vector bundle with Hermitian metric $h$ on $X$.  
We assume that $c_1 (E) =0$. 
We denote by $\mathcal{A}_{(E,h)}$ the space of all connections on $E$ which preserve the metric $h$, 
and by $\mathfrak{u} (E) = \text{End}(E,h)$ the bundle of skew-Hermitian endomorphisms 
of $E$.

In these setting,  we have $d_{A} = \partial_{A} +
 \bar{\partial}_{A}$, 
$d_{A}^{*} = \partial_{A}^{*} + 
 \bar{\partial}_{A}^{*}$ and
$\mathfrak{a} 
= \phi - \phi^{*}$, where $\phi \in \Gamma (X, \mathfrak{u} (E) \otimes \Omega^{1}_{X} ) 
= \Omega^{1 ,0} ( \mathfrak{u} (E))$ with $\Omega^{1}_{X}$ being the holomorphic cotangent bundle of $X$. 
In addition, the space of complexified two-forms is decomposed as 
 $\Omega^{2} (X) \otimes \C 
= \Omega^{2,0} (X) \oplus \Omega^{1,1} (X) \oplus \Omega^{0,2} (X)$, 
and we have the following identifications.  
\begin{gather*} 
\Omega^{+} (X) \cong \Omega^2 (X) \cap 
\left( \Omega^{2,0}(X) \oplus \Omega^{0,2} (X) \oplus \Omega^0 (X)
 \omega \right) , \\
\quad \Omega^{-} (X) \cong \Omega^2 (X) \cap ( \Omega_{0}^{1,1}
 (X)), 
\end{gather*}
where $\Omega^{1,1}_{0} (X)$ denotes the orthogonal subspace in
$\Omega^{1,1} (X)$ to $\Omega^{0} (X)\omega$. 
Thus, the equation \eqref{eq:KW2} has the following form on a compact K\"{a}hler surface. 
\begin{equation}
F_{A}^{0,2} - [ \phi^{*} \wedge \phi^{*} ] = 0, \quad 
i \Lambda (F_{A}^{1,1} +2  [ \phi \wedge \phi^{*}] ) =0, 
\label{KWk1}
\end{equation}
where $\Lambda$ denotes the adjoint of
$\wedge \omega$

Furthermore, we have the following. 
\begin{proposition}
On a compact K\"{a}hler surface, the equations \eqref{eq:KW1} and \eqref{eq:KW2} have the following form that asks $(A, \phi) \in \mathcal{A}_{(E,h)} \times \Omega^{1,0 } 
(\mathfrak{u} (E))$ to satisfy 
\begin{gather}
\bar{\partial}_{A} \phi = 0 
 , \quad [ \phi \wedge \phi ] =0, 
\label{KWcK1}\\
F_{A}^{0,2} = 0, \quad \Lambda \left( F_{A}^{1,1} + 2 [ \phi  \wedge \phi^{*}] \right) =0 . 
\label{KWcK2} 
\end{gather}
\label{prop:Simp}
\end{proposition}

\vspace{-0.8cm}
Namely, the Kapustin--Witten equations and Simopson's equations in \cite{Si} are the same on a compact K\"{a}hler surface.  
Note that the above statement for the flat bundle case was proved by Simpson \cite{Si2} 
(in arbitrary dimensions).

\vspace{0.3cm}
\hspace{-0.8cm}
{\it Proof of Propositoin \ref{prop:Simp}
\footnote{The proof presented here was partly inspired by some proficient calculations by Clifford Taubes \cite{Tp}.}.} 
The proof consists of the following three steps.

\vspace{0.2cm}
\hspace{-0.6cm}\underline{Step 1}:  
As  $d_{A}^{*} = \partial_{A}^{*} + 
 \bar{\partial}_{A}^{*}$ and
$\mathfrak{a} 
= \phi -  \phi^{*}$ with $\phi \in \Omega^{1 ,0} ( \mathfrak{u} (E))$, 
from the first equation $d_{A}^{*} \mathfrak{a} = 0$ in \eqref{eq:KW1}, we have
$ \partial_{A}^{*} \phi - \bar{\partial}_{A}^{*} \phi^{*} =0 $.  
From this with the K\"{a}hler identities: $i \partial_{A}^{*} = [ \bar{\partial}_{A}
, \Lambda]$ and $i \bar{\partial}_{A}^{*} = - [ \partial_{A}
, \Lambda]$, we obtain 
$i \Lambda ( \bar{\partial}_{A} \phi + \partial_{A} \phi^{*} ) =0$. 
Hence, the equations in \eqref{eq:KW1} take the following form when $X$ is a 
compact K\"{a}hler surface. 
\begin{equation}
i \Lambda ( \bar{\partial}_{A} \phi + \partial_{A} \phi^{*} ) =0, \quad 
( \bar{\partial}_{A} \phi -  \partial_{A} \phi^{*} )^{-} = 0.
\label{eq2} 
\end{equation}

\vspace{0.2cm}
\hspace{-0.6cm}\underline{Step 2}:  
We next prove that 
\begin{equation}
\bar{\partial}_{A} \phi -  \partial_{A} \phi^{*} = 0,  \quad \Lambda \bar{\partial}_{A} \phi 
= \Lambda \partial_{A} \phi^{*} =0 .
\label{eq:step2}
\end{equation}

\vspace{0.1cm}
To prove \eqref{eq:step2}, we write the second equation in \eqref{eq2} in the following form.  
\begin{equation}
(\bar{\partial}_{A} \phi -  \partial_{A} \phi^{*} ) -  (\Lambda \bar{\partial}_{A} \phi ) \wedge \omega 
 + ( \Lambda \partial_{A} \phi^{*} ) \wedge \omega =0 
\label{eq:asd2}
\end{equation}
Acting on the left hand side of this by $\partial_{A}$, we get the following. 
\begin{equation*}
\begin{split}
0 &= \partial_{A} \left( 
(\bar{\partial}_{A} \phi -  \partial_{A} \phi^{*} ) -   (\Lambda \bar{\partial}_{A} \phi ) \wedge \omega 
 +  ( \Lambda \partial_{A} \phi^{*} ) \wedge \omega 
\right)  \\
& =  \partial_{A} \bar{\partial}_{A} \phi  - \partial_{A} \partial_{A} \phi^{*}  
 - \partial_{A} (\Lambda \bar{\partial}_{A} \phi ) \wedge \omega 
 + \partial_{A} ( \Lambda \partial_{A} \phi^{*} ) \wedge \omega  \\
& =  \partial_{A} \bar{\partial}_{A} \phi  - \partial_{A} \partial_{A} \phi^{*}  
 -( \Lambda \partial_{A} \bar{\partial}_{A} \phi ) \wedge \omega  \\ 
& \qquad + i ( \bar{\partial}_{A}^{*}  \bar{\partial}_{A} \phi ) \wedge \omega  
+ ( \Lambda \partial_{A} \partial_{A} \phi^{*} ) \wedge \omega  
     - i  ( \bar{\partial}_{A}^{*}  \partial_{A} \phi^{*} ) \wedge \omega  \\
& =  i ( \bar{\partial}_{A}^{*}  \bar{\partial}_{A} \phi ) \wedge \omega  
 - i  ( \bar{\partial}_{A}^{*}  \partial_{A} \phi^{*} ) \wedge \omega  . 
\end{split}
\end{equation*}
Here we used the K\"{a}hler identities at the third equality, and the fact that the multiplication by $\omega$ is an isometry from 1-forms to 3-forms and the $\Lambda$ is the adjoint of it at the last equality above. 
Then the $L^2$ inner product of this with $\phi$ yields 
\begin{equation}
i  \langle  \phi,   \bar{\partial}_{A}^{*}  \bar{\partial}_{A} \phi     \rangle_{L^2}
- i \langle \phi,   \bar{\partial}_{A}^{*}  \partial_{A} \phi^{*}  \rangle_{L^2} =0 . 
\label{eq:l21}
\end{equation}
Similarly, acting on the left hand side of \eqref{eq:asd2} by $\bar{\partial}_{A}$, we obtain 
\begin{equation*}
\begin{split}
0 &= \bar{\partial}_{A}  \left( 
(\bar{\partial}_{A} \phi -  \partial_{A} \phi^{*} ) -  (\Lambda \bar{\partial}_{A} \phi ) \wedge \omega 
 + ( \Lambda \partial_{A} \phi^{*} ) \wedge\omega 
\right)  \\
& =  \bar{\partial}_{A} \bar{\partial}_{A} \phi  - \bar{\partial}_{A} \partial_{A} \phi^{*}  
 - \bar{\partial}_{A} (\Lambda \bar{\partial}_{A} \phi ) \wedge \omega 
 + \bar{\partial}_{A} ( \Lambda \partial_{A} \phi^{*} ) \wedge \omega  \\
& =  \bar{\partial}_{A} \bar{\partial}_{A} \phi  - \bar{\partial}_{A} \partial_{A} \phi^{*}  
 -( \Lambda \bar{\partial}_{A} \bar{\partial}_{A} \phi ) \wedge \omega  \\ 
& \qquad -  i ( \partial_{A}^{*}  \bar{\partial}_{A} \phi ) \wedge \omega  
+ ( \Lambda \bar{\partial}_{A} \partial_{A} \phi^{*} ) \wedge\omega  
     +  i  ( \partial_{A}^{*}  \partial_{A} \phi^{*} ) \wedge \omega   \\
& = -  i ( \partial_{A}^{*}  \bar{\partial}_{A} \phi ) \wedge \omega  
  +  i  ( \partial_{A}^{*}  \partial_{A} \phi^{*} ) \wedge \omega 
\end{split}
\end{equation*}
Then from the $L^2$ inner product of this with $\phi^*$ we get  
\begin{equation}
- i \langle \phi^*,   \partial_{A}^{*}  \bar{\partial}_{A} \phi \rangle_{L^2}  
 + i  \langle  \phi^*,   \partial_{A}^{*}  \partial_{A} \phi^{*}   \rangle_{L^2} =0 . 
\label{eq:l22}
\end{equation}
Then adding \eqref{eq:l21} and \eqref{eq:l22}, we obtain  
\begin{equation*}
||  \bar{\partial}_{A} \phi - \partial_{A} \phi^{*} ||^{2}_{L^2} =0.  
\end{equation*}
Thus $\bar{\partial}_{A} \phi - \partial_{A} \phi^{*} =0$. 
Substituting this to the first term of \eqref{eq:asd2}, we get  
$ (\Lambda \bar{\partial}_{A} \phi ) \omega 
  - ( \Lambda \partial_{A} \phi^{*} ) \omega = 0$. 
Thus from this and  the first equation in \eqref{eq2}, we obtain $ \Lambda \bar{\partial}_{A} \phi  
  = \Lambda \partial_{A} \phi^{*}  = 0$.

\vspace{0.2cm}
\hspace{-0.6cm}\underline{Step 3}:  
We now prove that $\bar{\partial}_{A} \phi =0$ and $[ \phi \wedge \phi ] = 0$;  and consequently $F_{A}^{0,2}=0$.

Acting on the first equation in \eqref{eq:step2} by $\bar{\partial}_{A}^*$, we get 
$$ \bar{\partial}_{A}^* \bar{\partial}_{A} \phi  - \bar{\partial}_{A}^* \partial_{A} \phi^{*} = 0
$$ 
Taking the $L^2$ inner product of this with $\phi$, we obtain 
\begin{equation}
\langle \phi , \bar{\partial}_{A}^* \bar{\partial}_{A} \phi  \rangle_{L^2}  
-  \langle \phi, \bar{\partial}_{A}^* \partial_{A} \phi^{*}  \rangle_{L^2} = 0. 
\label{eq:step31}
\end{equation}
Here, the second term of the above can be computed as follows. 
\begin{equation*}
\begin{split}
\langle \phi, \bar{\partial}_{A}^* \partial_{A} \phi^{*}  \rangle_{L^2}  
&= i \langle \phi,  \partial_{A} \Lambda  \partial_{A} \phi^{*}  \rangle_{L^2}  
 - i \langle \phi,  \Lambda \partial_{A} \partial_{A} \phi^{*}  \rangle_{L^2}   \\
&= - i \langle \phi,  \Lambda \partial_{A} \partial_{A} \phi^{*}  \rangle_{L^2}   \\ 
&= - i \langle \omega \wedge  \phi,  [F_{A}^{2,0} \wedge \phi^{*} ]  \rangle_{L^2}   \\
&= - i \langle \omega \wedge \phi,  [[\phi \wedge\phi] \wedge \phi^{*} ]  \rangle_{L^2}   \\
&= - \int_{X} \text{tr}  ( [[\phi \wedge\phi] \wedge \phi^{*} ]  \wedge \phi^{*} ) \\
&= - \int_{X} \text{tr}  ( [\phi \wedge\phi] \wedge [\phi^{*} \wedge \phi^{*} ] ) \\
&= - || [ \phi \wedge \phi ]  ||_{L^2}^{2} . 
\end{split}
\end{equation*}
In the above, we used the K\"{a}hler identity at the first equality; the fact that $\Lambda \partial_{A} \phi^* =0$ from \eqref{eq:asd2} at the second equality; and the conjugate of the first equation in \eqref{KWk1} at the forth equality. 
Thus \eqref{eq:step31} becomes 
$$ || \bar{\partial}_{A} \phi ||_{L^2}^{2} + || [ \phi \wedge \phi ] ||_{L^2}^{2} = 0. 
$$
Hence the assertion holds. 
\qed

\begin{remark}
Nakajima \cite[\S6(iii)]{N} obtained the complex form of the equations described above in an elegant way even including the Vafa--Witten equations case. 

\end{remark}

\subsection{The structure of singular sets in the K\"{a}hler case}
\label{sec:stsing}
 
From here, we assume the rank of $E$ to be two. 
Let $\{ (A_n , \phi_n) \}_{n \in \N}$ be a sequence of solutions to the
equations \eqref{KWcK1} and \eqref{KWcK2}, and put $r_{n} := ||
\phi_{n} ||_{L^2}$ and $\varphi_{n} := \phi_{n} / r_n$ for each $n \in
\N$. 
Then the analysis by Taubes described in Section \ref{sec:rtaubes}  
to this article holds for solutions to
the equations \eqref{KWcK1} and \eqref{KWcK2}.  
In particular, if we assume that $\{ r_n \}$ has no bounded subsequence, 
then there exist 
a finite set of points in $X$ to be denoted by $\Theta$, 
a closed nowhere dense subset $Z$, 
a real line bundle $\mathcal{J}$ on $X \setminus \{ \Theta \cup Z \}$, 
a section $\mu$ of the bundle $\mathcal{J} \otimes \Omega^{1}_{X}$, 
a Hermitian vector bundle $E_{\Delta}$ with hermitian metric $h_{\Delta}$ on $X \setminus \{  \Theta \cup Z \}$, 
a connection $A_{\Delta}$ of $E_{\Delta}$, 
and a isometric bundle homomorphism $\tau_{\Delta} : \mathcal{J} \to 
\mathfrak{u} (E_{\Delta})$. 
In addition, there exist a subsequence $\Xi \in \N$ and a sequence of isomorphism $\{ g_{i} \}_{i \in \Xi}$ from $E_{\Delta}$ 
to $E|_{X \setminus \{ \Theta \cup Z\}} $ such that 
$\{ g_{i}^* \varphi_{i} \} $ converges to $\tau_{\Delta} \circ \mu$ in the $L^2_1$ topology on compact subsets in $X \setminus \{ \Theta \cup Z \} $; and in the $C^{0}$ topology on $X \setminus \Theta$, and 
$\{ g_{i} ^* A_{i} \} $ converges on compact subsets of $X \setminus \{ \Theta \cup Z \}$ to $A_{\Delta}$ in the $L^2_1$ topology.  Here $Z$ is the zero set of the $L^2_1$ limit $|\hat{\varphi}_{\diamond}|$ of a
subsequence $\{ |\varphi_{i}| \}$, which is obtained in the same way to 
$|\hat{a}_{\diamond}|$ in Section \ref{sec:rtaubes}.

As mentioned at the beginning of this section, the above singular set $Z$ has the  structure of an analytic 
subvariety in $X$ in this case. 
To see that,  we consider\footnote{The idea of taking this section was pointed out to the author by Clifford Taubes, 
also the proof of Proposition \ref{prop:ZT} below uses his ideas throughout. }  
a section $\text{tr} ( \varphi_i \otimes \varphi_i)$
of  $( \Omega_{X}^{1})^{\otimes 2} $ for each $\varphi_{i} \in \Gamma (\mathfrak{u} (E) \otimes \Omega_{X}^{1})$.  
Note that the connection on the bundle does not do anything to this as we take the trace. 
We then have the following.

\begin{proposition}
Assume that $\{ r_n \}$ has no bounded subsequence. 
Then there exists a subsequence $\Lambda \subset \N$ such that $\{ \text{\rm tr} ( \varphi_i \otimes \varphi_i)
 \}_{i \in \Lambda}$ converges in $C^{\infty}$-topology to a holomorphic section, which we denote
 by $\text{\rm tr} ( \hat{\varphi}_{\diamond} \otimes \hat{\varphi}_{\diamond})$, of the holomorphic bundle 
 $ ( \Omega_{X}^1 )^{\otimes 2}$. 
\label{prop:det}
\end{proposition}

\begin{proof}
From the definitions of $\varphi_{i}$ and $\text{tr} ( \varphi_i \otimes \varphi_i)$,  we get 
$$| \text{tr} ( \varphi_i \otimes \varphi_i) | \leq | \varphi_{i} |^2 \leq C . $$
Thus, $\{ \text{tr} ( \varphi_i \otimes \varphi_i)\}_{i \in \N}$ has a
 convergent subsequence. 
The regularity follows since 
$\bar{\partial} ( \text{tr} ( \varphi_i \otimes \varphi_i) ) = 0$ as $\bar{\partial}_{A_{i}}
 \varphi_{i} =0$ for each $i \in \N$. 
\end{proof}

\begin{remark}
The map $\text{tr} ( \varphi_i \otimes \varphi_i)$ is called the {\it Hitchin map} in the literature of the Higgs bundles, originally introduced in \cite{Hi}.  
It also appears as the Hopf differential of a harmonic map such as in \cite{DDW}.   
\end{remark}

We denote by $T$ the zero set of the section $\text{\rm tr} ( \hat{\varphi}_{\diamond} \otimes \hat{\varphi}_{\diamond})$. 
We then prove the following.

\begin{proposition}
$Z = T \setminus \Theta'$, where $\Theta' \subset 
\Theta$. 
\label{prop:ZT}
\end{proposition}

\begin{proof}
We obviously have $ Z \subset T$. 
In order to prove the opposite inclusion, 
we assume that there exists a point $p \in T \setminus \{ \Theta \cup Z\} $. 
We then recall the following inequality in \cite{Hi}, which holds for $2 \times 2$
 trace free matrices 
 $\Phi$. 
\begin{equation} 
|\Phi|^4 \leq 4 | \det \Phi |^2 
 + | [ \Phi , \Phi^{*}] |^2 . 
\label{eq:ineqPhi}
\end{equation}
With the above inequality in mind, we prove that $p \in Z  = | \hat{\varphi}_{\diamond}|^{-1}
 (0)$. 
We take a local orthonormal coframe $\{ \hat{e}^1 , \hat{e}^2 \}$ of $\Omega_{X}^{1}$ around $p$ to write 
$ \varphi_{i} = \mathfrak{c}_{i,1} \hat{e}^{1} + \mathfrak{c}_{i, 2} \hat{e}^2$. 
We view $\mathfrak{t}_{i} := \text{tr} ( \varphi_{i} \otimes \varphi_{i} ) $ as a symmetric matrix with components $\{ \mathfrak{t}_{i, \alpha \beta} \}_{\alpha, \beta = 1, 2}$. 
The $\mathfrak{t}_{i, 11}$ component is $\text{tr} ( \mathfrak{c}_{i,1} {}^2)$. 
Since $\text{tr} ( \mathfrak{c}_{i, 1}) =0$, we have $\mathfrak{t}_{i ,11} = -2 \det (\mathfrak{c}_{i,1})$. 
As $p \in T \setminus \{Z \cup \Theta \}$, 
there exists a number $N_{1} \in \Lambda$ such that for $i \geq N_{1}$, 
we have $  | 2 \det (\mathfrak{c}_{i,1} )| (p) < \varepsilon$ for a given
 $\varepsilon >0$.

Next, as described in Section \ref{sec:rtaubes}, 
$\{ \varphi_{i} \}$ converges  in $L^2_1$ topology on compact subsets of $X \setminus \{ \Theta
 \cup Z \}$ to $\tau_{\Delta} \circ \mu$ after gauge
 transformations. 
Furthermore, this convergence is in $C^{0}$ topology on compact subsets
 of $X \setminus \Theta$. 
This implies in particular $\Lambda [ \varphi_{i} , \varphi_{i}^{*}]$ converges to $0$ in
 $C^{0}$ topology on compact subsets of $X \setminus \{ \Theta \cup Z \}$,  
since $\tau_{\Delta} \circ \mu$ is abelian in $\mathfrak{u} (E_{\Delta} )$. 
We take a compact subset $B$ of $X \setminus \{ \Theta \cup Z \}$, 
which contains the point $p$. 
Then one concludes that the limit $ \Lambda [ \hat{\varphi}_{\diamond} ,
 \hat{\varphi}_{\diamond}^{*} ] $ is identically zero on the whole of $B$. 
Thus, for a given $\varepsilon >0$, there exists a number  $N_2 \in \Lambda$ such that 
$| \Lambda [ \varphi_{i} , \varphi_{i}^{*} ] | (p)  = | [ \mathfrak{c}_{i,1} , \mathfrak{c}_{i,1}^{*} ]  + 
[ \mathfrak{c}_{i,2}  , \mathfrak{c}_{i,2}^{*} ] | (p) < \varepsilon$ for $i \geq
 N_2$. 
On the other hand, from the second equation in \eqref{KWcK1}, we have 
$[ \mathfrak{c}_{i,1} , \mathfrak{c}_{i,2} ] =0$, and both $\mathfrak{c}_{i,1}$ and $\mathfrak{c}_{i,2}$ are trace free, we obtain that 
$\mathfrak{c}_{i,2} = \alpha  \mathfrak{c}_{i,1}$ with $\alpha$ being a complex number. 
Thus we have 
$ | ( 1 + |\alpha|^2) [ \mathfrak{c}_{i,1}  , \mathfrak{c}_{i,1}^{*} ] | (p) < \varepsilon$ for $i \geq N_{2}$.

Summarizing these above with the inequality \eqref{eq:ineqPhi}, 
we find an $N' \in \Lambda$ for a given $\varepsilon >0$ such
 that $|\mathfrak{c}_{i,1} | (p) < \varepsilon$  for all $i \geq N'$.  
A similar argument does for $\mathfrak{c}_{i ,2}$.  
Hence $p \in Z = |\hat{\varphi}_{\diamond}|^{-1} (0)$. Thus the assertion holds.  
\end{proof}

Since $\text{\rm tr} ( \hat{\varphi}_{\diamond} \otimes \hat{\varphi}_{\diamond})$ 
is a holomorphic section of the holomorphic  
bundle $( \Omega^1_{X})^{\otimes 2}$ on $X$,  $T$ has the structure of an analytic subvariety. 
We thus obtain the following.

\begin{Corollary}
$Z$ has the structure of an analytic subvariety of $X$. 
\label{cor:hs}
\end{Corollary}

\begin{example}
We give simple examples of the above $T$ for the case that $X$ is the direct product of two Riemann surfaces. 
(i) When $X= \mathbb{P}^1 \times \mathbb{P}^1$, there are no non-trivial $\phi \in \Gamma ( \mathfrak{u} (E) \otimes \Omega^{1}_{X})$ satisfying the equations 
from the beginning.  
(ii) When $X = \mathbb{P}^1 \times T^2$, then $\Omega^{1}_{X} 
\cong \Omega^{1}_{\mathbb{P}^1} \oplus \Omega^{1}_{T^2} \cong K_{\mathbb{P}^1} \oplus \underline{\C}$. 
Thus, the holomorphic section $\text{\rm tr} ( \hat{\varphi}_{\diamond} \otimes \hat{\varphi}_{\diamond})$ of $( \Omega^{1}_{X} )^{\otimes 2}$ has the form 
$( \underline{0} \oplus \underline{a} ) \otimes ( \underline{0} \oplus \underline{b} )$, 
where $a, b \in \C$ with either $a$ or $b$ being non-zero. 
Since $a \neq 0$ or $b \neq 0$, thus $Z = \emptyset$.  
(iii) If $X = \mathbb{P}^1 \times \Sigma_{g}$ then $\Omega_X^1 \cong \Omega_{\mathbb{P}^1} ^1 \oplus \Omega_{\Sigma_g}^1 \cong 
K_{\mathbb{P}^1} \oplus K_{\Sigma_g}$; 
and $\text{tr} ( \hat{\varphi}_{\diamond} \otimes \hat{\varphi}_{\diamond} )$ of $(\Omega_X^1)^{\otimes 2}$ has the form $(\underline{0} \oplus s ) \times (\underline{0} \oplus t )$, where $s,t \in \Gamma( \Sigma_{g})$. 
Hence $Z = ( \mathbb{P}^1 \times s^{-1} (0) ) \cap ( \mathbb{P}^1 \times t^{-1} (0))$. 
This is generically an empty set. 
(iv) For the case that $X = T^2 \times T^2$, $\text{\rm tr} ( \hat{\varphi}_{\diamond} \otimes \hat{\varphi}_{\diamond})$ has the form 
$( \underline{a} \oplus \underline{b}) \otimes ( \underline{c} \oplus \underline{d})$, 
where $a, b, c ,d \in \C$. 
Since at least one of $a, b, c, d$ can not be zero, so $Z = \emptyset$. 
(v) When $X = T^2 \times \Sigma_{g}$, where $\Sigma_{g}$ is a Riemann surface with genus $g > 1$, 
then $\text{\rm tr} ( \hat{\varphi}_{\diamond} \otimes \hat{\varphi}_{\diamond})$ has the form $( \underline{a} \oplus s  ) \otimes ( \underline{b} \oplus t)$, where $a, b \in \C$ and $s, t \in \Gamma (K_{\Sigma_{g}})$.  
Thus, if $a=0$ and $b=0$, $Z = (T^2 \times s^{-1} (0) ) \cap (T^2 \times t^{-1} (0) )$; 
and otherwise $Z = \emptyset$. 
(vi) When $X = \Sigma_{g} \times \Sigma_{h}$ with $g, h >1$, 
then 
$\text{\rm tr} ( \hat{\varphi}_{\diamond} \otimes \hat{\varphi}_{\diamond})$ 
has the form $(s \oplus t) \otimes (v \oplus w)$, where $s, v \in \Gamma (K_{\Sigma_{g}} )$, 
$t , w \in \Gamma (K_{\Sigma_{h}})$. 
Then $Z$ is $(s^{-1} (0) \times t^{-1} (0) ) \cap ( v^{-1}(0) \times w^{-1}(0))$. 
For instance, 
if $s$ and $t$ are identically zero; and one of $v$ or $w$ is identically zero,  then 
$Z$ is either $\Sigma_g \times w^{-1} (0)$ or $v^{-1} (0) \times \Sigma_h $. 
Or, if $v$ and $w$ are identically zero; and one of $s$ and $t$ is identically zero, 
then $Z$  is either  
$\Sigma_g \times t^{-1} (0)$ or $s^{-1} (0) \times \Sigma_h$. 
\qed
\end{example}

\section{The singular sets of solutions to the Vafa--Witten equations}
\label{sec:VW}

\subsection{Results by Taubes}
\label{sec:rtaubesVW}

The Vafa--Witten equations look similar to the Kapustin--Witten equations, but one of the crucial differences for us is that there is no good control of the curvatures of connections. 
However, Taubes managed to prove the convergence of the Higgs fields outside a singular set in a similar style to the case of the Kapustin--Witten equations. 
Here we briefly state some of  his results. 
Firstly, in the same way to the Kapustin--Witten case, Taubes obtained the following.

\begin{proposition}[\cite{T4}]
Let $\{ (A_n , B_n ) \}$ be a sequence of solutions to the Vafa--Witten equations. 
Assume that $|| B_n ||_{L^2}$ diverges as $n$ goes to the infinity. 
We rescale $B_n$ by $|| B_n ||_{L^2}$, namely, put $\beta_n  := B_n / || B_n||_{L^2}$ for each $n \in \mathbb{N}$. 
Then $\{ \beta_n \}$ has a converging subsequence in $L^2_1$ and $C^0$ topologies on compact subsets of $X$. 
\end{proposition}

We denote by $| \hat{\beta}_{\diamond}|$ the limit and define a closed subset $Z'$ of $X$ by the zero set of $| \hat{\beta}_{\diamond}|$. 
Despite the fact that no obvious control of the curvature of connections, 
Taubes proved the following. 

\begin{theorem}[\cite{T4}]
The Hausdorff dimension of $Z'$ is at most two, and there exists a real line bundle $\mathcal{I}$ on $X \setminus Z'$; 
a section $\nu$ of $\mathcal{I} \otimes \Lambda^+$ with $d \nu =0$ and $|\nu| = |\hat{\beta}_{\diamond}|$; and 
a sequence of isometric homomorphisms $\{ \sigma_n \}$ from $\mathcal{I}$ to $\mathfrak{g}_P |_{X \setminus Z'}$ such that $\{ \beta_n - \sigma_n \circ \nu \}$ converges to zero in $C^0$-topology on compact subsets of $X$. 
\end{theorem}

\subsection{The equation on compact K\"{a}hler surfaces and the singular sets}
\label{sec:singVW}

As in the case of the Kapustin--Witten equations, we have the complex form of the equations, when the underlying manifold is a compact  K\"{a}hler surface (see \cite[Ch.7]{M} or \cite[\S  6(iii)]{N}). 
The exact form is as follows: 
Let $X$ be a compact K\"{a}hler surface, and let $E \to X$ be a Hermitian vector bundle of rank $r$ over $X$.  
Then the Vafa--Witten equations \eqref{VW1}, \eqref{VW2} become the following equations seeking for a pair $(A, \phi)$ consisting of a connection $A$ of $E$ and a section $\phi$ of $\mathfrak{u} (E) \otimes K_X$, where $K_X$ is the canonical bundle of $X$. 
\begin{gather}
\bar{\partial}_{A} \phi =0 , 
\label{VWk1} \\
F_A^{1,1} \wedge \omega + [ \phi , \phi^*] = 0, \quad F_A^{0,2} =0 . 
\label{VWk2} 
\end{gather}
Note that $[\phi \wedge \phi ]=0$ automatically holds as $K_X$ is a line bundle.

Then almost the same argument works for the Vafa--Witten case as well,  even in a simpler way.  
Let us consider a sequence of solutions $\{ (A_n , \phi_n) \}_{n \in \N}$ to the equations \eqref{VWk1}, \eqref{VWk2}. 
We also assume here that the rank of $E$ is two. 
We are interested in the case that $|| \phi ||_{L^2}$ diverges as $n$ goes to the infinity. 
So put $r_n := || \phi_n ||_{L^2}$ for each $n \in \mathbb{N}$ and suppose that $\{ r_n \}_{n \in \N}$ has no converging subsequence. 
We then put $\varphi_n := \phi_n / || \phi_n ||_{L^2}$ for each $n \in \N$ and consider, 
in this case, the determinant $\det \varphi_n$, 
which is a section of $K_{X}^{\otimes 2}$ for the rank two case. 
As in the case of the Kapustin--Witten equations, we firstly get the following. 

\begin{lemma}
There exists $\Lambda \subset \N$ such that 
$\{ \det \varphi_i \}_{i \in \Lambda} $ converges in $C^{\infty}$-topology to a holomorphic section $\det \varphi_{\diamond}$ of $K_X^{\otimes 2}$. 
\end{lemma}

We denote by $D$ the zero set of $\det \varphi_{\diamond}$. 
Then we have the following. 

\begin{proposition}
$Z' = D$. 
\end{proposition}

\begin{proof}
The proof goes in the same way as in the case of the Kapustin--Witten equations (Proposition \ref{prop:ZT}). 
Namely, we use the inequality 
\begin{equation} 
|\Phi|^4 \leq 4 | \det \Phi |^2 
 + | [ \Phi , \Phi^{*}] |^2 . 
\end{equation}
for $2 \times 2$
 trace free matrices $\Phi$ again.  
For the Vafa--Witten case, we can directly use this inequality to obtain a bound on $| \varphi_i|$ in terms of those on $| \det \varphi_i |$ and $[\varphi_{i} , \varphi_{i}^{*}]|$. 
We omit the repetition of following the argument in the proof of Proposition \ref{prop:ZT} here. 
\end{proof}

Hence we obtain the following. 

\begin{Corollary}
$Z'$ has the structure of an analytic subvariety of $X$.  
\label{cor2}
\end{Corollary}


\begin{flushleft}
Mathematical Institute, University of Oxford \\
Radcliffe Observatory Quarter, Woodstock Road, Oxford, OX2 6GG, U.K.\\
tanaka@maths.ox.ac.uk
\end{flushleft}


\end{document}